\documentclass[a4paper,12pt]{article}
\usepackage[utf8]{inputenc} 
\usepackage{amsmath}
\usepackage{amsthm}
\usepackage{amssymb}
\usepackage{mathrsfs}

\usepackage[icelandic,english]{babel}
\usepackage{t1enc}
\usepackage[colorlinks=false,
pdfauthor={Benedikt Steinar Magnusson and Erlend Fornaess Wold},
pdftitle={Continuous Carleman approximations
}]{hyperref}

\newcommand{\C}{{\mathbb  C}}
\newcommand{\R}{{\mathbb  R}}

\renewcommand{\phi}{\varphi}
\renewcommand{\epsilon}{\varepsilon}

\newtheorem{theorem+}           {Theorem}      [section]
\newtheorem{definition+}  [theorem+]  {Definition}
\newtheorem{lemma+}  [theorem+]  {Lemma}
\newtheorem{corollary+}  [theorem+]  {Corollary}
\newtheorem{proposition+}  [theorem+]  {Proposition}
\newtheorem{example+}  [theorem+]  {Example}
\newtheorem{question+}  [theorem+]  {Question}

\newenvironment{theorem}{\begin{theorem+}\sl}{\end{theorem+}\rm}
\newenvironment{definition}{\begin{definition+}\rm}{\end{definition+}\rm}
\newenvironment{lemma}{\begin{lemma+}\sl}{\end{lemma+}\rm}
\newenvironment{corollary}{\begin{corollary+}\sl}{\end{corollary+}\rm}

\title{A characterization of totally real Carleman continua and an application 
to products of stratified totally real sets.  
}
\author{Benedikt Steinar Magnusson and Erlend Forn{\ae}ss Wold}
\date{\today}

\begin{document}
\maketitle
\begin{abstract}
We give a characterization of stratified totally real sets that admit Carleman 
approximation by entire functions.  As an application we
show that the product of two stratified totally real Carleman continua is a Carleman continuum.  
\end{abstract}


\section{Introduction}

\begin{theorem}\label{thm4}
Let $M\subset\mathbb C^n$ be a closed stratified totally real set.  Then $M$ 
is a Carleman continuum if and only if $M$ has bounded E-hulls. 
\end{theorem}

(For the definition of Carleman approximation and bounded E-hulls, see Section 2, and 
for the definition of a stratified totally real set, see Section 3.)

It is already known \cite{ManOvrWol:2011} that in the totally real setting, the property of bounded E-hulls 
implies Carleman approximation.  The remaining result is therefore the following, which does 
not rely on the notion of being totally real:

\begin{theorem}\label{thm1}
 If $M\subset \C^n$ is a closed set which admits 
Carleman approximation by entire functions, then $M$ has bounded E-hulls.
\end{theorem}

Theorem \ref{thm1} generalizes the main theorem of \cite{ManOvrWol:2011} where the implication 
was shown under the assumption that $M$ is totally real and admits $\mathcal C^1$-Carleman approximation.   
As an application of this theorem we prove the following partial answer to a question raised by E. L. Stout (private communication):

\begin{theorem}\label{thm2}
Let $M_j\subset\mathbb C^{n_j}$ be stratified totally real sets for $j=1,2$ which 
admit Carleman approximation by entire functions.  Then $M_1\times M_2\subset\mathbb C^{n_1}\times\mathbb C^{n_2}$ 
admits Carleman approximation by entire functions.  
\end{theorem}
(For the definition of a stratified totally real set see Section 3)

\medskip

The precise question is more general: \emph{if 
$M_j$ are Carleman continua in $\mathbb C^{n_j}$ for $j=1,2,$ is $M_1\times M_2$ Carleman?} \

\medskip

A natural generalization of one of the main results of \cite{ManOvrWol:2011} which 
we will use in the proof of Theorem \ref{thm2} is the following:

\begin{theorem}\label{thm3}
Let $M\subset\mathbb C^n$ be a stratified totally real set which has bounded E-hulls, 
and let $K\subset\mathbb C^n$ be a compact set such that $K\cup M$
is polynomially convex.   Then any $f\in\mathcal C(K\cup M)\cap\mathcal O(K)$
is approximable in the Whitney $\mathcal C^0$-topolgy by entire functions. 
\end{theorem}

\medskip

As a corollary to Theorem \ref{thm4} we also obtain the following:
\begin{corollary}
Let $M\subset\mathbb C^n$ be a totally real manifold of class $\mathcal C^k$ and 
assume that $M$ admits Carleman approximation by entire functions.   Then 
$M$ admits $\mathcal C^k$-Carleman approximation by entire functions.  
\end{corollary}
\begin{proof}
By \cite{ManOvrWol:2011} this follows from the fact that $M$ has bounded E-hulls.  
\end{proof}

For a short survey on the topic of Carleman approximation, see \emph{e.g.} the introduction of 
\cite{ManOvrWol:2011}.

\section{Proof of Theorem \ref{thm1}}

\begin{definition}
Let $X\subset\mathbb C^n$ be a closed set.   We say that $X$ is a Carleman continuum, or that 
$X$ admits Carleman approximation by entire functions, if $\mathcal O(\mathbb C^n)$ is 
dense in $\mathcal C(M)$ in the Whitney $\mathcal C^0$-topology.  
\end{definition}

\begin{definition}
Let $X\subset\mathbb C^n$  be a closed subset.  Given a compact normal exhaustion $X_j$ of $X$ 
we define the polynomial hull of $X$, denoted by $\widehat X$, by $\widehat X:=\cup_j\widehat X_j$
(this is independent of the exhaustion).  We also set $h(X):=\widehat X\setminus X$.    
If $h(X)$ is empty we say that $X$ is polynomially convex.  
\end{definition}
\begin{definition}
We say that a  closed set $M\subset\mathbb C^n$ has bounded E-hulls if 
for any compact set $K\subset\mathbb C^n$ the set $h(K\cup M)$ is bounded.  
\end{definition}

We give two lemmas preparing for the proof of Theorem \ref{thm1}.
The first one is a simple well known result \'{a} la Mittag-Leffler and Weierstrass, which 
we state for the lack of a suitable reference. 

\begin{lemma}\label{lemma:f_l}
Let $E=\{x_j\}_{j\in\mathbb N}$ be a discrete sequence in $\C^n$. 
Then for any sequence $\{a_j\}_{j\in\mathbb N}\subset\mathbb C$
there exists an entire function $f\in\mathcal O(\mathbb C^n)$
with $f(x_j)=a_j$ for all $j\in\mathbb N$, and
there exist holomorphic functions $f_1,...,f_n$
such that $E=Z(f_1,...,f_n)$. 
\end{lemma}
\begin{proof}
By Theorem 3.7 in \cite{RosayRudin1988} there exists an injective 
holomorphic map $F=(\tilde f_1,...,\tilde f_n):\mathbb C^n\rightarrow\mathbb C^n$ such 
that $F(x_j)=j\cdot{\bf e_1}$.   For the first claim we let $g\in\mathcal O(\mathbb C)$ be an entire 
function with $g(j)=a_j$ for all $j\in\mathbb N$ and set $f=g\circ \tilde f_1$.  
For the second claim let  $g\in\mathcal O(\mathbb C)$ be an entire 
function whose zero set is precisely $\{j\}_{j\in\mathbb N}$.  
Now set $f_1=g\circ\tilde f_1$ and $f_j=\tilde f_j$ for $j=2,...,n$.
\end{proof}

\begin{lemma}\label{lemma:approx}
 Let $M\subset\mathbb C^n$ be a Carleman continuum 
 and let 
 $E=\{x_j\}_{j\in\mathbb N}$ be a discrete set of points with $E\subset\mathbb C^n\setminus M$.  
Then $M\cup E$ is a Carleman continuum, with interpolation on $E$.  
 \end{lemma}
\begin{proof}
Assume that $q:M \cup E \to \C$ and $\epsilon: M \cup E \to \R_+$
are continuous functions.
By Lemma \ref{lemma:f_l} there exist functions $f_1,..,f_n\in\mathcal O(\mathbb C^n)$
such that $f_j(x_k)=0$ for all $x_k\in E$ and $j=1,...,n$, and such 
that $Z(f_1,...,f_n)\cap M=\emptyset$.  So there exist continuous functions 
$g_j\in\mathcal C(M)$ such that 
$$
g_1\cdot f_1+\cdot\cdot\cdot+g_n\cdot f_n=1
$$
on $M$.  
Since $M$ admits Carleman approximation we may approximate 
the $g_j$'s by entire functions $\tilde g_j\in\mathcal O(\mathbb C^n)$
such that the function 
$$
\phi=\tilde g_1\cdot f_1+\cdot\cdot\cdot+\tilde g_n\cdot f_n
$$
satisfies $\phi(x)\neq 0$ for all $x\in M$.   Obviously $\phi(z)=0$
for all $z\in E$. \

By the Mittag-Leffler Theorem there exists an entire function $h\in\mathcal O(\mathbb C^n)$
such that $h(z)=q(z)$ for all $z\in E$.  Let $\psi\in\mathcal C(M)$ be 
the function $\psi(x):=\frac{h(x)-q(x)}{\phi(x)}$.    Since $M$ admits 
Carleman approximation we may approximate $\psi$ by en entire function 
$\sigma\in\mathcal O(\mathbb C^n)$, and if the approximation 
is good enough, the function $h(z)-\phi(z)\cdot\sigma(z)$ is 
$\epsilon$-close to $q$ on $M\cup E$.  
\end{proof}

\emph{Proof of Theorem \ref{thm1}:}
Aiming for a contradiction we assume that 
$M$ does not have bounded exhaustion hulls, \emph{i.e.}, there exists a compact set $K$
such that $h(K\cup M)$ is not bounded.  
This implies that there is a discrete sequence of points 
$E=\{x_j\} \subset h(K\cup M)$. 
By Lemma \ref{lemma:approx} there exists an entire function  $q\in\mathcal O(\mathbb C^n)$ such that 
$|q(z)| < \frac 12$, for $z\in M$ and $q(x_j)=j$ for $x_j\in E$. 
Define $C=\|q\|_K$.  For $j>C$ we then have that $|q(j)|>\underset{z\in K\cup M}{\sup}\{|q(z)|\}$ 
which contradicts the assumption that $E\subset h(K\cup M)$.

\section{Proof of Theorem \ref{thm3}}

\begin{definition}
Let $M\subset\mathbb C^n$ be a closed set.  We say that $M$ is a stratified totally 
real set if $M$ is the increasing union $M_0\subset M_2\subset\cdot\cdot\cdot\subset M_N=M$
of closed sets, such that $M_j\setminus M_{j-1}$ is a totally real set for $j=1,...,N$, and 
with $M_0$ totally real.
\end{definition}

The proof of Theorem \ref{thm3} is an inductive construction depending on the following lemma (\cite{SamuelssonWold2012}, Theorem 4.5)
\begin{lemma}\label{inductivestep}
Let $K\subset\mathbb C^n$ be a compact set, let $M\subset\mathbb C^n$ be a compact stratified 
totally real set, and assume that $K\cup M$ is polynomially convex.  Then 
any function $f\in\mathcal C(K\cup M)\cap\mathcal O(K)$ is uniformly approximable 
by entire functions.  
\end{lemma}
\begin{proof}
Set $X_j:=K\cup M_j$ for $j=0,...,N$.   Then $X_0$ is polynomially convex (see the proof of Theorem 4.5 in \cite{SamuelssonWold2012})
and so it follows from \cite{ManOvrWol:2011} that $f|_{X_0}$ is uniformly approximable by entire 
functions.  The result is now immediate from Theorem 4.5 in \cite{SamuelssonWold2012}.
\end{proof}

\emph{Proof of Theorem \ref{thm3}:} Choose a normal exhaustion 
$K_j$ of $\mathbb C^n$ such that $K_j\cup M$ is polynomially convex for 
each $j$.   We choose $K_1=K$.  Assume that we are given $f\in\mathcal C(K\cup M)\cap\mathcal O(K)$ and 
$\epsilon\in\mathcal C(K\cup M)$ with $\epsilon(x)>0$ for all $x\in M$.   We will 
construct an approximation of $f$ by induction on $j$, and so 
we assume that we have constructed $f_j\in\mathcal C(K_j\cup M)\cap\mathcal O(K_j)$
such that $|f_j(x) - f(x)|<\epsilon(x)/2$ for all $x\in M$.
Choose $\chi\in\mathcal C^\infty_0(K_{j+2}^\circ)$ with $\chi\equiv 1$ near $K_{j+1}$.  
Given any $\delta_j>0$ it follows from Lemma \ref{inductivestep} that 
there exists an entire function $g_j$ such that $|g_j(x)-f_j(x)|<\delta_j$ 
for all $x\in K_j\cup (M\cap K_{j+1})$.   We set $f_{j+1}:=g_j$ on $K_{j+1}$
and $f_{j+1}:=\chi\cdot g_j + (1-\chi)(f_j)$ on $M\setminus K_{j+1}$.  
It is clear that we get  
\begin{itemize}
\item[($j_1$)] $\|f_{j+1}-f_j\|_{K_j}<\delta_j$, 
\end{itemize}
and we have on $M_{j+1}\setminus M_{j}$ that $f_{j+1}-f=(f_j-f)+\chi\cdot(g_j - f_j)$, so
if $\delta_j$ is sufficiently small we also get that 
\begin{itemize}
\item[($j_2$)] $|f_{j+1}(x)-f(x)|<\epsilon(x)/2$ for all $x\in M$.
\end{itemize}
If $\delta_j$ decreases fast enough it follows from $j_1$ that 
$f_j$ converges to an entire function $\tilde f$, and it 
follows from $j_2$ that $|\tilde f(x)-f(x)|<\epsilon(x)$ for 
all $x\in M$.

\section{Proof of Theorem \ref{thm2}}

Note first that $M_1\times M_2$ is a stratified totally real set, so by 
Theorem \ref{thm3} it suffices to show that $M_1\times M_2$
has bounded E-hulls.  
Let $K\subset\mathbb C^{n_1}\times\mathbb C^{n_2}$ be compact.   Since 
both $M_1$ and $M_2$ are Carleman continua, they have bounded 
exhaustion hulls by Theorem \ref{thm1}.     Choose compact 
sets $\tilde K_j$ in $\mathbb C^{n_j}$ for $j=1,2,$ with $K\subset \tilde K_1\times \tilde K_2$.
Set $K_j:=\tilde K_j\cup \overline{h(\tilde K_j\cup M)}$ which are now compact sets.  

We claim that $(K_1\times K_2)\cup (M_1\times M_2)$ is polynomially convex, from which 
it follows that $h(K\cup (M_1\times M_2))\subset K_1\times K_2$.    Let $(z_0,w_0)\in (\mathbb C^{n_1}\times\mathbb C^{n_2})\setminus [ (K_1\times K_2)\cup (M_1\times M_2)]$.  We consider several cases.
\begin{itemize}
\item[i)] $z_0\notin K_1\cup M_1$.  Here we simply use a function in the variable $z_0$ only.
\item[ii)] $w_0\notin K_2\cup M_2$.  Analogous to i). 
\item[iii)] $z_0\in M_1\cap K_1$.  Then $w_0\notin K_2\cup M_2$, so we are in case ii).
\item[iv)] $z_0\in M_1\setminus K_1$.  Then $w_0\in K_2\setminus M_2$, unless we are in case ii). 
By Lemma \ref{lemma:approx}  there exists $f\in\mathcal O(\mathbb C^{n_2})$ such that 
$f(w_0)=1$ and $|f(w)|<1/2$ for all $w\in M_2$.   Set $N=\|f\|_{K_2}$.  By Lemma \ref{thm3}
there exists $g\in\mathcal O(\mathbb C^{n_1})$ such that $\|g\|_{K_1}<1/(2N), |g(z)|<3/2$ for all 
$z\in M_1$ and $g(z_0)=1$.  Set $h(z,w)=f(w)\cdot g(z)$.   Then $h(z_0,w_0)=1$.   
For $(z,w)\in K_1\times K_2$ we have $|h(z,w)|=|f(w)||g(z)|\leq N\cdot 1/(2N)=1/2$.
If $(z,w)\in M_1\times M_2$ then $|h(z,w)|=|f(w)||g(z)|\leq 1/2\cdot 3/2=3/4$.  
\item[v)] $z_0\in K_1\setminus M_1$.   Then $w_0\in M_2\setminus K_2$ unless we are in case ii), but 
this is the same as iv) with the roles of $z_0$ and $w_0$ switched.  
\end{itemize}
$\hfill\square$

\end{document}